\documentclass{amsart}

\usepackage{amsmath}
\usepackage{amsthm,thmtools}
\usepackage{amssymb,amsbsy}
\usepackage{tikz-cd}
\usepackage{hyperref}
\usepackage{enumitem}
\setlist{nosep}
\setlist[enumerate,1]{label=\emph{(}\alph*\emph{)}\,,align=right}
\usepackage{ytableau}
\usepackage{xcolor}
\usepackage{caption}
\newcommand{\NN}{\mathbb{N}}
\let\SS\relax
\newcommand{\SS}{\mathbb{S}}

\newcommand{\RR}{\mathbb{R}}

\let\C\relax
\newcommand{\C}{\mathcal{C}}

\let\O\relax
\newcommand{\O}{\mathcal{O}}
\let\P\relax
\newcommand{\P}{\mathcal{P}}

\newcommand{\V}{\mathcal{V}}

\newcommand{\tC}{\mathtt{C}}

\newcommand{\op}{\mathrm{op}}

\newcommand{\monop}{\mathbb{N}^{\times,\op}_+}

\newcommand{\comm}{\mathsf{Comm}}

\newcommand{\sh}{\mathsf{Sh}}
\newcommand{\psh}{\mathsf{PSh}}

\newcommand{\pts}{\mathsf{Pts}}

\newcommand{\loc}{\mathsf{Loc}}
\newcommand{\topo}{\mathsf{Top}}

\newcommand{\ofp}{\mathrm{fp}}

\newcommand{\y}{\mathbf{y}}

\newcommand{\pro}{\mathrm{pro}}

\newcommand{\cl}{\mathrm{cl}}

\newcommand{\bcl}{\mathbf{cl}}
\newcommand{\bcoh}{\mathbf{coh}}
\newcommand{\bep}{\mathbf{ep}}
\newcommand{\bgt}{\mathbf{gt}}
\newcommand{\bp}{\mathbf{p}}
\newcommand{\bx}{\mathbf{x}}

\newcommand{\join}{\vee}

\newtheorem{definition}{Definition}
\newtheorem{proposition}[definition]{Proposition}

\newtheorem{theorem}[definition]{Theorem}
\newtheorem{corollary}[definition]{Corollary}
\newtheorem{lemma}[definition]{Lemma}
\newtheorem{example}[definition]{Example}
\newtheorem{remark}[definition]{Remark}

\tikzset{
b/.style={bend left=10},
bb/.style={bend left},
cl/.style={outer sep=-1pt},
}

\let\phi\varphi
\let\theta\vartheta
\let\epsilon\varepsilon

\newcommand\twoheaddownarrow{{\rotatebox[origin=c]{-90}{$\twoheadrightarrow$}}}

\definecolor{llgray}{gray}{0.9}
\definecolor{lgray}{gray}{0.8}

\title{Grothendieck topologies on posets}
\author{Jens Hemelaer}
\thanks{The author is a Ph.D.\ fellow of the Research Foundation -- Flanders (FWO)}
\address{Department of Mathematics, University of Antwerp \\ 
 Middelheimlaan 1, B-2020 Antwerp (Belgium) \\ {\tt jens.hemelaer@uantwerpen.be}}
\begin{document}

\begin{abstract}
Lindenhovius has studied Grothendieck topologies on posets and has given a complete classification in the case that the poset is Artinian. We extend his approach to more general posets, by translating known results in locale and domain theory to the study of Grothendieck topologies. In particular, explicit descriptions are given for the family of Grothendieck topologies with enough points and the family of Grothendieck topologies of finite type. As an application, we compute the cardinalities of these families in various examples.
\end{abstract}

\maketitle
\tableofcontents

\section{Introduction}

When studying a category from a geometrical point of view, it is often desirable to have a concrete description of a certain Grothendieck topology on the category. This concrete description includes for example determining the points for such a Grothendieck topology, or answering whether or not the associated topos is coherent or subcanonical.

One example is the category $\C = \comm_\ofp^\op$, the opposite category of the category of finitely presented commutative rings. In \cite{gabber-kelly}, very explicit descriptions are given for the points of various Grothendieck topologies.\footnote{Note that Gabber and Kelly work with sheaves on the category of separated schemes of finite type (and relative versions thereof). However, this is equivalent to sheaves on $\comm_\ofp^\op$ whenever the Grothendieck topology is finer than the Zariski topology \cite[Proof of Remark 2.4]{gabber-kelly}.} As a demonstration of why this is useful, they point out an application to sheaf cohomology \cite[Proposition 4.5]{gabber-kelly}. For this category $\comm_\ofp^\op$, it is already an open problem to give a complete description for the points of the flat topology (some partial results are in \cite[Lemma 3.3]{gabber-kelly} and in \cite{schroeer}).

However, if the category $\C$ is a poset, it is much easier to describe the Grothen\-dieck topologies on it. The reason is that $\psh(\C)$ is equivalent to the category of sheaves on a topological space $X$, and Grothendieck topologies are in bijection to sublocales of $\O(X)$ (the locale of opens of $X$).

In this paper, we exploit this fact to describe the Grothendieck topologies on posets. Of course, locales, posets, and their interactions are a very well-studied topic, so all the necessary ingredients are already in the literature. Our contribution is translating these results to topos theory, and in this way extending the results of \cite{lindenhovius}.

The underlying idea of this work is that a better understanding of Grothendieck topologies on posets can help us understanding Grothendieck topologies on more general categories. For example, in \cite[Proposition 2]{llb-covers}, it is shown that any Grothendieck topology on the monoid $\tC = \monop$ (the underlying monoid for the Arithmetic Site of Connes and Consani), comes from a Grothendieck topology on a certain poset (the big cell). So in this case, a classification of Grothendieck topologies on a related poset leads to a classification of Grothendieck topologies on the original category of interest. Similarly, in \cite{azureps} and \cite{azutop}, Grothendieck topologies on the big cell are used to study Grothendieck topologies on the category of Azumaya algebras an center-preserving algebra maps. Here an important part is determining which topologies on the big cell correspond to coherent subtoposes. We will show in Section \ref{sect:coherent} how to approach this problem for a very general class of posets.

\section{Presheaves on a poset} \label{sect:presheaves-on-a-poset}

Posets $P$ will always be identified with the corresponding category that has
\begin{itemize}
\item as objects the elements of $P$;
\item a unique morphism $p \to q$ whenever $p \leq q$.
\end{itemize}
This is the same convention as in \cite{lindenhovius}. We will later define the dcpo of filters on $P$. There is an embedding $P^\op \subseteq X$, so it can be useful to keep in mind that the ordering in $X$ is opposite to the ordering in $P$.

It is well-known that a category of sheaves on a poset is a localic topos. The category of \emph{presheaves} on a poset in addition has enough points, so it is equivalent to $\sh(X)$ for some topological space $X$. Moreover, we can assume that this space $X$ is sober, and then it is uniquely determined. Indeed, the elements of $X$ can be identified with the topos points of $\psh(P)$ and the topology is the subterminal topology, see \cite[2.1]{caramello-stone}.

The category of points of $\psh(P)$ is equivalent to $(P_\pro)^\op$, so clearly it is again a poset. We can describe $(P_\pro)^\op$ explicitly as the poset of filters on $P$ under the inclusion relation. In particular, we identify the elements of $X$ with the filters on $P$. Subterminal objects in $\psh(\C)$ are in bijection with the downwards closed subsets in $P$. For a downwards closed subset $U$ of $P$, the corresponding open set for the subterminal topology is the set $\tilde{U}$ containing all filters intersecting $U \subseteq P$. For each $U$, there is a set of minimal filters contained in $U$, namely the set of principal filters associated to the generators of $U$. So we can write the open sets as
\begin{equation} \label{eq:ideal-notation}
(p_i)_{i \in I} = \{ x \in X : \exists i\in I,~ p_i \leq x  \} \subseteq X,
\end{equation}
where $\{p_i\}_{i \in I}$ is an indexed family of principal filters, and $\leq$ denotes the partial order on the filters (i.e.\ inclusion of filters).\footnote{The notation $(p_i)_{i \in I}$ is inspired by the ideal notation for rings. Caveat: poset ideals are downwards closed, so $(p_i)_{i \in I}$ is almost never a poset ideal.}

Notice that filters on $P$ are exactly the same thing as ideals on $P^\op$, so this gives another point of view. In fact, this point of view is the most common in domain theory: for a poset $P$, the poset of ideals of $P$ with the inclusion relation is called the \emph{ideal completion}. So we can describe the category of points of $\psh(P)$ as the ideal completion of $P^\op$. Ideal completions are so-called \emph{algebraic dcpo's}, with the principal ideals as \emph{finite elements} \cite[Proposition 5.1.46]{goubault-larrecq}.

\begin{definition}[{\cite[Section 5.1]{goubault-larrecq}}]
Let $X$ be a poset. Then for $a,b \in X$ we write
\begin{equation}
a \ll b
\end{equation}
if and only if $b \leq \sup(D)$ implies $\exists d \in D,~ a \leq d$, for any directed subset $D\subseteq X$ admitting a supremum.\footnote{A directed subset $D$ is a subset for which any finite collection of elements in $D$ has an upper bound in $D$. In particular, it is nonempty.} In this case, we say that $a$ is \emph{way below} $b$ (this clearly implies $a \leq b$). We write
\begin{equation}
\twoheaddownarrow x = \{ y \in X : y \ll x  \}.
\end{equation}
We say that $x \in X$ is \emph{finite} if $x \ll x$. We say that $X$ is \emph{algebraic} if for each $x \in X$ the set of finite elements smaller than $x$ is directed with supremum $x$. We call a subset $B \subseteq X$ a \emph{basis} if for each $x \in X$, the set $B \cap \twoheaddownarrow x$ is directed with supremum $x$. We say that $X$ is \emph{continuous} if for each $x \in X$ the set $\twoheaddownarrow x$ is directed with supremum $x$. We say that $X$ is a \emph{dcpo} if any nonempty directed subset of $X$ admits a supremum.
\end{definition}

It is easy to show that a poset $X$ is algebraic if and only if the set of finite elements is a basis, and it is continuous if there exists a basis. In particular, every algebraic poset is continuous \cite[Lemma 5.1.23, Corollary 5.1.24]{goubault-larrecq}.

\begin{definition}[{The Scott topology, \cite[Section 4.2]{goubault-larrecq}}]
Let $X$ be a poset. Then a subset $U \subseteq X$ is called Scott-open if it is upwards closed and if for any directed subset $D \subseteq X$ with $\sup(D) \in U$ we have $U \cap D \neq \varnothing$.

For algebraic posets, the Scott topology has as basis of open sets the sets
\[
\uparrow x = \{ y \in X : x \leq y \}
\]
where $x$ ranges over the finite elements of $X$ \cite[Theorem 5.1.27]{goubault-larrecq}.

The specialization order associated to the Scott topology is the original partial order on $X$.\footnote{Here we define the specialization order on the points of a topological space as $x \leq y$ if and only if there is an inclusion of point closures $\cl(x)\subseteq \cl(y)$. The opposite convention is popular as well.}
\end{definition}

We can now summarize the above results with the following corollary. The content of this corollary can be found e.g.\ in \cite[Subsection 10.2]{amadio-curien}, but no topos-theoretic language is used there. On the other hand, in \cite[Subsection 4.2]{caramello-stone}, the relation to topos theory is discussed and \cite[Proposition 4.1]{caramello-stone} is very close to the corollary below.

\begin{corollary}[See also {\cite[Proposition 4.1]{caramello-stone}}.]
There is an equivalence of categories between:
\begin{enumerate}
\item the category of localic presheaf toposes and geometric morphisms between them (considered upto natural isomorphism);
\item the category of algebraic dcpo's and Scott continuous morphisms between them.\footnote{Scott continuous maps can be defined in two equivalent ways: either as monotonous maps preserving directed suprema, or literally as maps that are continuous for the Scott topology \cite[Theorem 7.3.1(iv)]{vickers}.}
\end{enumerate}
\end{corollary}
\begin{proof}
If a presheaf topos $\psh(\C)$ is localic, then $\C$ embeds (contravariantly) in the category of points of $\psh(\C)$, so $\C = P$ is a poset. Let $X$ be the algebraic dcpo of filters on $P$. Conversely, any algebraic dcpo $X$ is the category of points for $\psh(F^\op)\simeq \sh(X)$ where $F \subseteq X$ is the subset of finite elements.

The equivalence of categories now follows by the fact that the category of sober topological spaces and continuous maps is a full subcategory of the category of toposes and geometric morphisms (considered upto natural isomorphism) \cite[IX.3, Corollary 4]{maclane-moerdijk-sheaves} \cite[IX.5, Proposition 2]{maclane-moerdijk-sheaves}.
\end{proof}

\section{Grothendieck topologies versus subsets}

Let $P$ be a poset, and let $X$ be the dcpo of filters on $P$. Then by the results of the previous section $\psh(P) \simeq \sh(X)$, where $X$ is equipped with the Scott topology. The Grothendieck topologies on $P$ are in bijective correspondence with the subtoposes of $\psh(P) \simeq \sh(X)$, which in turn are in bijective correspondence with the sublocales of $\O(X)$ \cite[IX.5, Corollary 6]{maclane-moerdijk-sheaves}. In particular, the Grothendieck topologies with enough points correspond to the spatial sublocales, which are in bijective correspondence with the sober subspaces of $X$ \cite[IX.3, Corollary 4]{maclane-moerdijk-sheaves}.

If $X$ is a sober space, then its sober subspaces are precisely the closed subspaces for a topology called the \emph{strong topology} \cite[Corollary 3.5]{keimel-lawson}.

\begin{definition}[{\cite{keimel-lawson},\cite[Exercise V-5.31]{compendium-2}}] Let $X$ be a topological space. Then we say that $Y \subseteq X$ is \emph{locally closed} if it is the intersection of an open set with a closed set. The locally closed sets clearly form a basis for a new topology, which we call the \emph{strong topology} on $X$.

If $X$ with the original topology is $T_0$, then the strong topology is the one generated by the original topology and the lower sets in the specialization order.
\end{definition}

\begin{corollary}[{Corollary of \cite[IX.3, Corollary 4]{maclane-moerdijk-sheaves}}] \label{cor:correspondence-ep-subspace} Let $P$ be a poset and let $X$ be its dcpo of filters (with the Scott topology). Then there is an adjunction\footnote{Adjunctions between categories that are posets are usually called Galois connections.}
\begin{equation}
\begin{tikzcd}[column sep = large]
\P(X) \ar[r,b,"{K_{(-)}}"] &
\mathrm{GT}(P)^\op \ar[l,b,"{S_{(-)}}"]
\end{tikzcd}
\end{equation}
between the poset $\P(X)$ of subsets of $X$ (and inclusions), and the opposite of the poset $\mathrm{GT}(P)$ of Grothendieck topologies on $P$ (and inclusions).

This adjunction restricts to an equivalence
\begin{equation}
\V(X) ~\simeq~ \mathrm{GT}_{\mathrm{ep}}(P)^\op
\end{equation}
where $\V(X)^\op$ denotes the full subcategory of subspaces that are closed for the strong topology, and $\mathrm{GT}_{\mathrm{ep}}(P)$ denotes the full subcategory of Grothendieck topologies with enough points. Moreover, $\V(X)$ and $\mathrm{GT}_{\mathrm{ep}}(P)^\op$ are the images of $S_{(-)}$ and $K_{(-)}$, respectively.
\end{corollary}
\begin{proof}
A similar adjunction is given in \cite[IX.3, Corollary 4]{maclane-moerdijk-sheaves}. There it is proven that the two functors
\begin{equation}
\begin{tikzcd}
\O ~:~ \topo~ \ar[r,shift left=1] & \ar[l,shift left=1] ~\loc ~:~ \pts
\end{tikzcd}
\end{equation}
are adjoint (with $\O$ the left adjoint). Here $\O(X)$ is the locale of opens of $X$, and $\pts(L)$ is the space of points of $L$, with the subterminal topology. The adjoint functors are equivalences if we restrict to the category of locales with enough points on the left, and the category of sober spaces on the right (see locus citus). 

In our situation, we can identify $\P(X)$ with the equivalence classes of embeddings $Y \hookrightarrow X$. Similarly, $\mathrm{GT}(P)^\op$ can be identified with the equivalence classes of locale embeddings $L \hookrightarrow \O(X)$. If $Y \hookrightarrow X$ is a topological embedding, one can check that $\O(Y) \hookrightarrow \O(X)$ is an embedding of locales. Conversely, if $L \hookrightarrow \O(X)$ is an embedding of locales, then $\pts(L) \hookrightarrow X$ is a topological embedding of the spaces of points.

The adjunction between $\topo$ and $\loc$ gives rise to a commuting diagram
\begin{equation}
\begin{tikzcd}
\loc(\O(Y),L) \ar[r,"{\sim}"] \ar[d,"{j \circ (-)}"] & \topo(Y,\pts(L)) \ar[d,"{j \circ (-)}"] \\
\loc(\O(Y),\O(X)) \ar[r,"{\sim}"] & \topo(Y,X)
\end{tikzcd},
\end{equation}
for $i: Y \hookrightarrow X$ a topological embedding and $j : L \hookrightarrow \O(X)$ a locale embedding. This induces a natural bijection between the fibers of $j \in \loc(\O(Y),\O(X)) \simeq \topo(Y,X)$ and this shows that $\O$ and $\pts$ restrict to an adjunction between $\P(X)$ and $\mathrm{GT}(P)^\op$. 

So we have proven the existence of the adjunction. The remaining statements follow directly from \cite[IX.3]{maclane-moerdijk-sheaves}, and the fact that adjunction is induced by the adjunction of \cite[IX.3, Corollary 4]{maclane-moerdijk-sheaves}.
\end{proof}

Take a sublocale $L \subseteq \O(X)$ corresponding to a Grothendieck topology $J$ on $P$. Then from the above proof it follows that $S_J = \pts(L) = \pts(\sh(P,J))$. Using \cite[Proposition 1.4]{gabber-kelly}, we can compute
\begin{equation} \label{eq:SJ}
S_J = \left\{ x \in X ~:~ p \leq x ~\Rightarrow~ x \in\!\! \bigcap_{L \in \Omega_J(p)} \!\widetilde{L}  \right\}.
\end{equation}
Here $\Omega_J(p)$ denotes the set of $J$-covering sieves on $p$, and $\tilde{L}$ denotes the upwards closure of $L$ in $X$ (in other words, the Scott open set with as generators the generators of $L$).

Now take $Y \subseteq X$. The corresponding locale embedding is defined by the frame homomorphism $U \mapsto U \cap Y$. So the covering sieves $L$ on $p$ are the ones such that $\tilde{L} \cap (p)$ contains $Y$, where $\tilde{L}$ is the upwards closure of $L$ in $X$ (or in other words, the Scott open set corresponding to $L$). So we can write
\begin{equation} \label{eq:KY}
K_Y(p) = \left\{\vphantom{\int} L \text{ sieve on }p\text{ such that }\tilde{L}\supseteq (p)\cap Y  \right\}.
\end{equation}
This gives a concrete description for the functors in Corollary \ref{cor:correspondence-ep-subspace}.

\section{Grothendieck topologies of finite type} \label{sect:coherent}

A Grothendieck topology is called \emph{of finite type} if every covering sieve contains a finitely generated covering sieve. In this section, we will relate Grothendieck topologies of finite type to coherent subtoposes and to patches (as introduced by Hochster in \cite{hochster}).

We recall some definitions from \cite[Subsection 7.3]{johnstone-topos}. An object $E$ in a topos is called \emph{compact} if any epimorphic family with codomain $E$ contains a finite epimorphic subfamily; the object $E$ is called \emph{coherent} if it is compact and any diagram $E' \rightarrow E \leftarrow E''$ of compact objects has a compact pullback.\footnote{In \cite{sga4-2}, ``compact'' is called ``quasi-compact''. Similarly, in this paper we use the word ``compact'' for (not necessarily Hausdorff) topological spaces such that every open cover admits a finite subcover.}

\begin{definition}[{\cite[Expos\'e vi, 2.4.5 and D\'efinition 3.1]{sga4-2}}] \label{def:coherent-topos} A \emph{coherent topos} is a topos containing a full subcategory $K$ of compact generators, that is moreover closed under finite limits. In this case, the objects of $K$ are all coherent.

A \emph{coherent geometric morphism} is a geometric morphism between coherent to\-poses such that the inverse image functor preserves coherent objects. We say that a subtopos of a coherent topos is a \emph{coherent subtopos} if it is coherent and if the inclusion geometric morphism is coherent as well.
\end{definition}

As remarked by Johnstone in \cite[Remark 7.48]{johnstone-topos}, one family of coherent toposes is given by the categories of sheaves $\sh(X)$ where $X$ is a \emph{spectral space} in the sense of Hochster \cite{hochster}.

\begin{definition}[{\cite[Section 0 and 1]{hochster}}] \label{def:spectral} A topological space $X$ is called a \emph{spectral space} if
\begin{itemize}
\item[(S1)] it is sober;
\item[(S2)] it is compact;
\item[(S3)] its compact opens are closed under intersections and form a basis.
\end{itemize}
A continuous map $f : X \to Y$ between spectral spaces is called \emph{spectral} if the inverse image of a compact open is again compact open. A subspace $Y \subseteq X$ of a spectral space $X$ is called a \emph{spectral subobject} if it is a spectral space and the inclusion morphism is spectral.
\end{definition}

If $X$ is a sober space, then $\sh(X)$ is coherent if and only if $X$ is spectral \cite[Remark 7.48]{johnstone-topos}.\footnote{Note that Johnstone does not use the terminology ``localic toposes'' here, instead he calls them ``toposes satisfying (SG)''.}

If $X$ is an algebraic dcpo with the Scott topology, then the open sets can all be written as $(p_i)_{i \in I}$ in the notation of (\ref{eq:ideal-notation}); the elements $p_i$ are called generators. Such an open set is clearly compact if and only if it can be generated by a finite set (in this case, we say that they are \emph{finitely generated}). In particular, the compact opens form a basis. So $X$ is a spectral space if and only if $X$ itself is finitely generated and $(p)\cap(q)$ is finitely generated for any two opens $(p)$ and $(q)$ generated by a single element \cite[Theorem 2.4]{priestley}.

This is the case, for example, if the finite elements of $X$ are a join-semilattice with a least element $0$, because then $X = (0)$ and $(p)\cap(q) = (p \join q)$. Another example is when $X$ has only a finite number of finite elements. As a counter-example, take $X$ the dcpo of filters on $(\NN,\leq)$. Then $X$ can be identified with $\NN \cup \{0\}$, with compact Scott open sets given by $\downarrow m = \{n \in \NN : n \leq m\}$. So $X$ is not compact.

In terms of the poset $P$ this means:

\begin{corollary}[{\cite[Remark 7.48]{johnstone-topos},\cite[Theorem 2.4]{priestley}}]
Let $P$ be a poset and let $X$ be its dcpo of filters. Then the following are equivalent:
\begin{enumerate}
\item $\psh(P) \simeq \sh(X)$ is a coherent topos;
\item $X$ is a spectral space;
\item there is a finite list $t_1,\dots,t_k \in P$ such that for all $p \in P$ there is a $t_i \geq p$,
and the same holds for the subposets $\downarrow\!p\,\, \cap \downarrow\!q$ with $p,q \in P$.
\end{enumerate}
\end{corollary}

Hochster's paper \cite{hochster} is well-known for proving that every spectral space is homeomorphic to the spectrum of a commutative ring \cite[Theorem 6]{hochster}. This explains the terminology. In order to reach this result, Hochster introduced the so-called \emph{patch topology}, which ``classifies'' spectral subobjects.

\begin{definition}[{\cite{hochster}}] Let $X$ be a spectral space. Then the \emph{patch topology} is the new (finer) topology on $X$ with as subbasis the compact opens of $X$ and their complements. A closed set for the patch topology is called a \emph{patch}.
\end{definition}

\begin{proposition} \label{prop:patches-coherent-oft}
Let $P$ be a poset and let $X$ be its dcpo of filters. We assume that $\psh(P) \simeq \sh(X)$ is a coherent topos, or equivalently, that $X$ is a spectral space. Let $S \subseteq X$ be a sober subspace. The following are equivalent:
\begin{enumerate}
\item $S \subseteq X$ is a patch;
\item $S \subseteq X$ is a spectral subobject;
\item $K_S$ is a Grothendieck topology of finite type;
\item $\sh(S) \subseteq \sh(X)$ is a coherent subtopos.
\end{enumerate}
\end{proposition}
\begin{proof}
$(a) \Leftrightarrow (b)$. This is in \cite[Section 2]{hochster}.

$(b) \Rightarrow (c)$. Let $L$ be a $K_S$-covering sieve on $p \in P$. In other words,
\begin{equation}
\widetilde{L} \cap (p) ~\supseteq~ S \cap (p).
\end{equation}
Now write $\widetilde{L} = (p_i)_{i \in I}$. Because $S \cap (p)$ is compact, we can find a finite subset $J \subseteq I$ such that the sieve corresponding to $(p_j)_{j \in J}$ is still a $K_S$-covering sieve.

$(c) \Rightarrow (d)$. Suppose that $K_S$ be a Grothendieck topology on $P$ of finite type. One can then show that $S\cap U$ is compact for every compact open $U \subseteq X$. The Yoneda embedding turns the compact opens $S \cap U \subseteq S$ into a generating set of compact objects, closed under pullbacks and containing the terminal object $S$ (this shows it is closed under all finite limits). So $\sh(S)$ is a coherent topos. We still have to show that the inclusion is coherent, i.e.~the pullback functor $i^* : \sh(X) \to \sh(S)$ preserves coherent objects. It is enough to show that $i^*(\y U) = \y(U \cap S)$ is coherent for $U$ a compact open of $X$. But $\y(U \cap S)$ is part of a generating family of compact objects, closed under finite limits. So it is coherent, see Definition \ref{def:coherent-topos}.

$(d) \Rightarrow (b)$. Suppose that $\sh(S) \subseteq \sh(X)$ is a coherent subtopos. Then $S$ is a spectral space by the remark above. We have to show that $U \cap S$ is compact open in $S$, for every compact open $U \subseteq X$. This follows by the assumption that the inclusion $\sh(S) \subseteq \sh(X)$ is coherent.
\end{proof}

We emphasize that the patch topology is finer than the (original) Scott topology, and coarser than the strong topology. In the situation of the above proposition: every closed set is a patch, and every patch is sober. Or: every closed subtopos is a coherent subtopos, and every coherent subtopos has enough points (Deligne's theorem).

\begin{table}[!ht]
\bgroup
\renewcommand{\arraystretch}{3}
\scalebox{0.85}{\begin{tabular}{c|c|c} 
Scott topology & Patch topology & Strong topology \\ \hline
\parbox[][][c]{4cm}{\centering ``Closed'' Grothendieck topologies} & 
\parbox[][][c]{4cm}{\centering Grothendieck topologies of finite type} & 
\parbox[][][c]{4cm}{\centering Grothendieck topologies with enough points} \\
\parbox[][][c]{4cm}{\centering Closed subtoposes} & 
\parbox[][][c]{4cm}{\centering Coherent subtoposes} & 
\parbox[][][c]{3cm}{\centering Subtoposes with enough points}
\end{tabular}}
\egroup
\label{table1}
\end{table}

It is clear from the definition of spectral subobjects that all finite subsets $S \subseteq X$ are patches. More examples will be given in the following section.

We end with a note on terminology. The Scott topology is called the \emph{localic topology} in the special case of the big cell in \cite{llb-covers}, \cite{azutop}. The strong topology on the big cell is called the \emph{pcfb-topology} in \cite{azutop}. Moreover, the strong topology is sometimes called the \emph{constructible topology} because it is the topology with the constructible sets as basis. Here a constructible set is defined as a finite union of locally closed sets. Sometimes in algebraic geometry \cite[D\'efinition 9.1.2]{ega3} a constructible set is defined as a finite union of subsets of the form $U \cap V^c$, where $U,V$ are so-called retrocompact opens. Then one could define the \emph{constructible topology} on $X$ (in the sense of Grothendieck) as the topology with the constructible sets (in the sense of Grothendieck) as basis \cite[\href{https://stacks.math.columbia.edu/tag/08YF}{Section 08YF}]{stacks-project}; in the case that $X$ is spectral this is precisely the patch topology on $X$. If $X$ is noetherian, then the strong topology and the patch topology agree, because in this case every open is compact. But in general the terminology ``constructible topology'' can be ambiguous, and it goes without saying that ``strong topology'' can mean different things as well (for example in \cite{nerode} it denotes what we call the patch topology\footnote{This is mentioned by Priestley in \cite{priestley}.}). Last but not least, in our situation the patch topology agrees with the so-called \emph{Lawson topology} on (special kinds of) posets, see \cite{compendium-2}, \cite{priestley}.

\section{More explicit description} \label{sect:more-explicit-description}

Let $P$ be a poset and $X$ its dcpo of filters, so $\psh(P)\simeq\sh(X)$. It is intuitively clear when a subset $S \subseteq X$ is closed for the Scott topology \cite[Remark II-1.4]{compendium-2}: $S$ is downwards closed and closed under directed suprema.

For the strong topology and the patch topology, it might be more difficult a priori to determine all closed sets.

Let $S \subseteq X$ be a subset closed for the strong topology. Take $s \notin S$. Then by definition we can find a (Scott) locally closed subset containing $s$, that does not intersect $S$. By shrinking the locally closed subset, we can assume it is of the form $(p) \cap \overline{s}$, where 
\begin{equation}
\overline{s} = \{ x ~:~ x \leq s \} \subseteq X
\end{equation}
is the Scott closure of $\{s\}$. In other words, there exists a $p \in P$, $p \leq s$, such that $S$ does not contain any $x$ with $p \leq x \leq s$. This statement is logically equivalent to the following:
\begin{equation}
(\forall p \in P \text{ with }p\leq s,~\exists x\in S\text{ with }p\leq x \leq s) ~\Rightarrow~ s \in S.
\end{equation}
Conversely, it easy to see that any subset with the above property is closed for the strong topology. So the strong topology is in some sense a topology expressing ``approximation from below''.

Now we look at the patch topology. We apply an idea of \cite{hochster} to our situation. Let $W$ be the Sierpinski space: it is the set $\{0,1\}$ with open sets $\varnothing, \{1\}, \{0,1\}$.\footnote{In comparison to the definition in \cite{hochster}, we swap $0$ and $1$. This is because Hochster uses the opposite definition of specialization order on a topology. We want that $0 \le 1$ for the specialization order.} For $X$ a topological space, the open subspaces of $X$ are precisely the sets $f^{-1}(1)$ for some continuous map $f : X \to W$. Recall that $X$ is spectral if and only if it is homeomorphic to a patch in $\prod_{i \in I} W$ for some index set $I$ \cite[Proposition 9]{hochster}.

Now we return to our case where $P$ is a poset and $X$ is its dcpo of filters that we assume to be spectral (for the Scott topology). Then we can say a little bit more. The finite elements of $X$ are just $P$, but with the opposite ordering. So we can consider the continuous map
\begin{equation}
\begin{tikzcd}
j~:~X \ar[r] & \prod_{p \in P} W
\end{tikzcd}
\end{equation}
defined by
\begin{equation}
x ~\mapsto~ \begin{cases}
1 \quad&\text{if }p \leq x, \\
0 \quad&\text{otherwise}.
\end{cases}
\end{equation}
in the component corresponding to $p \in P$. Here $\prod_{p \in P} W$ is itself the dcpo of filters on the poset $P'$ of finite subsets of $P$, with the opposite of the inclusion relation. The product topology agrees with the Scott topology. In the next section we will have a look at $\psh(P')$ in Example \ref{eg:power-set}.

We saw before that, in our case, $X$ is spectral if and only if it is compact and each $(p) \cap (q)$ is compact. The compact open sets in $\prod_{p \in P} W$ are the upwards closed sets generated by finitely many finite sets (we identify the elements of $\prod_{p \in P} W$ with subsets of $P$, the partial order is then inclusion of subsets). It is now straightforward to check that $j$ is a spectral map turning $X$ into a spectral subobject of $\prod_{p \in P} W$.

In the following we will always see $X$ as a patch in $\prod_{p \in P} W$ in the way described above. In particular:

\begin{lemma} \label{lmm:relative-patch-topology}
Let $X$ be a spectral space and let $S \subseteq X$ be a subset. Then $S$ is a patch in $X$ if and only if it is a patch in $\prod_{p \in P} W$. Moreover, $S$ is closed for the strong topology on $X$ if it is closed for the strong topology on $\prod_{p \in P} W$.
\end{lemma}
\begin{proof}
The last statement is clear.

We already mentioned that $f$ is spectral: if $U$ is a compact open of $\prod_{p \in P} W$, then $f^{-1}(U)$ is again compact open. We want to prove that the patch topology on $X$ is exactly the subspace topology with respect to the patch topology on $\prod_{p \in P} W$. It is enough to show that every compact open in $X$ can be written as $f^{-1}(U)$ for some compact open $U$. Use that $f^{-1}((\{p\})) = (p)$ for each $p \in P$, and the fact that taking inverse images preserves unions and intersections. Here $(\{p\})$ denotes the compact open set in $\prod_{p \in P} W$ generated by the singleton set $\{p\}$.
\end{proof}

\newpage

\begin{theorem}[{\cite[Theorem 2.1]{priestley}}] \label{thm:patch-vs-product} For each index set $I$, the patch topology on $\prod_{i \in I} W$ agrees with the product topology, where $W$ is equipped with the discrete topology.

When identifying $\prod_{i \in I} W$ with the set of all set-theoretic functions $I \to \{0,1\}$, the patch topology agrees with the topology of pointwise convergence.
\end{theorem}

\begin{definition} \label{def:pointwise-convergence}
Let $P$ be a poset and let $X$ be its dcpo of filters. Let $(x_i)_i$ be a sequence of elements in $X$. Then we say that $(x_i)_i$ \emph{converges pointwise} to an element $x \in X$ if for all $p \in P$, there is a natural number $N$ such that
\begin{equation}
p \leq x_i ~\Leftrightarrow~ p \leq x
\end{equation}
for all $i \geq N$.
\end{definition}

Then by Lemma \ref{lmm:relative-patch-topology} and Theorem \ref{thm:patch-vs-product}, a subset $S \subseteq X$ is a patch if and only if it is closed under pointwise convergence.

We say that $F \subseteq P$ is a \emph{separating set (of finite elements)} if for all $p \in P$, we can write $(p)$ as an intersection $(f_1)\cap\dots\cap(f_k)$, with $f_1,\dots,f_k \in F$. Then it is easy to see that the map $X \to \prod_{f \in F} W$ is injective, continuous and spectral. Moreover, Lemma \ref{lmm:relative-patch-topology} still holds if we replace $P$ by $F$ in the statement. To prove this we need that every compact open is the inverse image of a compact open in $\prod_{f \in F} W$, but this follows from $F$ being a separating set.

We end with an application.

\begin{proposition} \label{prop:patch-metrizable}
Let $P$ be a countable poset and $X$ its dcpo of filters. If $X$ is spectral, then the patch topology on $X$ is metrizable.
\end{proposition}
\begin{proof}
We can embed $X$ with the patch topology as a subspace of the space of set-theoretic functions $P \to \{0,1\}$ with the pointwise convergence, which is a metric space if and only if $P$ is countable.\footnote{The general result is that for a metric space $Y$, the space of continuous functions $X \to Y$ is metrizable if and only if $X$ is hemicompact. Further, $P$ with the discrete topology is hemicompact if and only if it is countable.}
\end{proof}

\begin{corollary}
Let $P$ be a countable poset such that $\psh(P)$ is a coherent topos. Then there is a metric space $X$ with its closed subsets in natural bijection with the Grothendieck topologies of finite type on $P$.

\end{corollary}

\section{Cardinalities of sets of Grothendieck topologies}

As an application of the explicit description from the previous section, we will compute the cardinalities of the sets of Grothendieck topologies on a poset (resp. Grothendieck topologies with enough points, Grothendieck topologies of finite type, Grothendieck topologies giving rise to closed subtoposes). Let $P$ be a poset and let $X$ be its dcpo of filters. We consider the topos $\psh(P) \simeq \sh(X)$. We use the following notations:
\begin{align*}
\bcl\qquad\text{\----}\qquad  &\text{cardinality of set of closed subtoposes} \\
\bcoh\qquad\text{\----}\qquad &\text{cardinality of set of coherent subtoposes} \\
\bep\qquad\text{\----}\qquad   &\text{cardinality of set of subtoposes with enough points} \\
\bgt\qquad\text{\----}\qquad   &\text{cardinality of set of Grothendieck topologies} \\
\bp\qquad\text{\----}\qquad    &\text{cardinality of }P \\
\bx\qquad\text{\----}\qquad    &\text{cardinality of }X \\
\end{align*}

\begin{proposition} \label{prop:cardinalities}
With the notations as above, suppose that $X$ is an infinite spectral space. Then in the table
\begin{align*}
\ytableausetup{mathmode, boxsize=3em}
\begin{ytableau}
\none & \none & \none & 2^{(2^\bp)} \\
\none & \none & \bgt  & 2^\bcl \\
\none & 2^\bp  & \bep  & 2^\bx \\
\bx   & \bcl  & \bcoh & 2^\bp \\
\end{ytableau}
\end{align*}
the cardinality in each box is less than or equal to the cardinalities in the boxes directly to the right of it and directly above it.
\end{proposition}
\begin{proof}
For the inequalities $\bx \leq \bcl$ use that all point closures are distinct (for the sober topology). The inequalities $\bcl \leq \bcoh \leq \bep \leq 2^\bx$ follow from the Scott topology being coarser than the patch topology, which is coarser than the strong topology, which is coarser than the discrete topology. The inequality $\bcl \leq 2^\bp$ follows from the fact that every Scott closed subset is determined by a set of elements in $P$. Further, each singleton $\{p\}$ with $p \in P$ is open for the strong topology, this shows $2^p \leq \bep$. For each $p \in P$, the set of finitely generated sieves on it has cardinality $\bp$ (each sieve being uniquely determined by its finite set of generators). So $\bcoh \leq \bp^\bp = 2^\bp$. An arbitrary Grothendieck topology is determined by a so-called nucleus, which is a function from the frame of opens to itself. So $\bgt \leq \bcl^\bcl = 2^\bcl$. Obviously, $\bep \leq \bgt$. The inequalities in the last column follow directly from $\bp \leq \bx$ and the inequalities that we already proved.
\end{proof}

\subsection{Artinian posets}

First suppose that $P$ is an \emph{Artinian poset} (every subset has a minimal element).\footnote{A poset that is Artinian is sometimes said to be \emph{well-founded}. The Artinian property is equivalent to the descending chain condition: every chain $p_1 \geq p_2 \geq p_3 \geq \dots$ eventually stabilizes.} This is the situation for which all Grothendieck topologies are explicitly described in \cite{lindenhovius}. In this case, all filters are principal, so $X = P^\op$. This also means that the open sets in $X$ are just the upwards closed sets (or equivalently downwards closed sets in $P$). By \cite[Theorem 10.1.13]{lindenhovius-thesis} every Grothendieck topology is of the form $K_S$ for some subset $S \subseteq P^\op = X$.\footnote{For subsets $S \subseteq P^\op$, the Grothendieck topology $J_S$ from \cite{lindenhovius} and \cite{lindenhovius-thesis} agrees with the Grothendieck topology $K_S$ as in (\ref{eq:KY}). The latter is inspired by the first one, with the only difference that it is defined for all subsets of $X$ rather than subsets of $P^\op \subseteq X$.}

Assume $P$ infinite. Using a diagram like in Proposition \ref{prop:cardinalities}: $\bp=\bx$ and in every gray box below, the cardinality is equal to $2^\bp$.
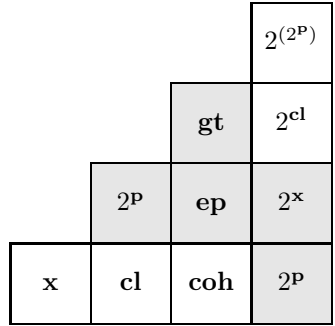
\begin{align*}
\ytableausetup{mathmode, boxsize=3em}
\begin{ytableau}
\none & \none & \none & 2^{(2^\bp)} \\
\none & \none & *(llgray) \bgt  & 2^\bcl \\
\none & *(llgray) 2^\bp  & *(llgray) \bep  & *(llgray) 2^\bx \\
\bx   & \bcl  & \bcoh & *(llgray) 2^\bp \\
\end{ytableau}
\end{align*}

\begin{example}[Almost discrete posets] Let $P$ be an infinite set containing a maximal element $1$, with $p \leq 1$ for each $p \in P$ as only relations. Then $P$ is clearly Artinian and $X$ is spectral. All subtoposes of $\psh(P)$ have enough points. There are clearly $2^\bp$ closed subspaces (equiv.~closed subtoposes).
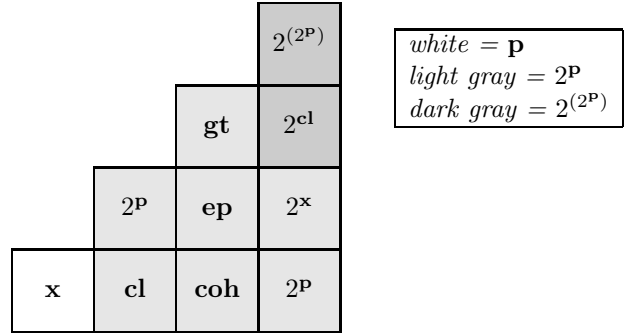
\begin{align*}
\ytableausetup{mathmode, boxsize=3em}
\begin{ytableau}
\none & \none & \none & *(lgray) 2^{(2^\bp)} \\
\none & \none & *(llgray) \bgt  & *(lgray) 2^\bcl \\
\none & *(llgray) 2^\bp  & *(llgray) \bep  & *(llgray) 2^\bx \\
\bx   & *(llgray) \bcl  & *(llgray) \bcoh & *(llgray) 2^\bp \\
\end{ytableau} & \qquad
\begin{array}{|l|} \hline
\text{white = } \bp \\ \text{light gray = } 2^\bp \\ \text{dark gray = } 2^{(2^\bp)} \\
\hline
\end{array}
\end{align*}
\end{example}

\begin{example}[The natural numbers] Let $P$ be the poset of natural numbers. In this case, there are countably many closed subtoposes (one for each natural number $n \in P$). So $\bp = \bx = \bcl$. Further, like in the previous example, singletons are open in the patch topology. So the patch topology agrees with the discrete topology, which shows $\bcoh = 2^\bp$.
\begin{align*}
\ytableausetup{mathmode, boxsize=3em}
\begin{ytableau}
\none & \none & \none & *(lgray) 2^{(2^\bp)} \\
\none & \none & *(llgray) \bgt  & *(llgray) 2^\bcl \\
\none & *(llgray) 2^\bp  & *(llgray) \bep  & *(llgray) 2^\bx \\
\bx   & \bcl  & *(llgray) \bcoh & *(llgray) 2^\bp \\
\end{ytableau} & \qquad
\begin{array}{|l|} \hline
\text{white = } \omega \\ \text{light gray = } 2^\omega \\ \text{dark gray = } 2^{(2^\omega)} \\
\hline
\end{array}
\end{align*}
The same computations can be done for $P$ an arbitrary infinite ordinal.
\end{example}

\begin{example}
Let $P$ be the poset of finite subsets of some infinite set $I$, with the inclusion relation. Then the cardinality of $P$ is equal to the cardinality of $I$. For each subset $I' \subseteq I$ we can define the downwards closed set
\begin{equation}
\left\{\vphantom{\int}  \{i\} ~:~ i \in I'  \right\} \cup \{ \varnothing \}.
\end{equation}
This shows that there are $2^\bp$ open sets (so $2^\bp$ closed sets as well).
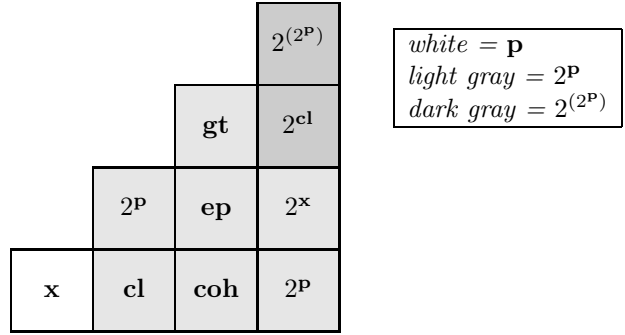
\begin{align*}
\ytableausetup{mathmode, boxsize=3em}
\begin{ytableau}
\none & \none & \none & *(lgray) 2^{(2^\bp)} \\
\none & \none & *(llgray) \bgt  & *(lgray) 2^\bcl \\
\none & *(llgray) 2^\bp  & *(llgray) \bep  & *(llgray) 2^\bx \\
\bx   & *(llgray) \bcl  & *(llgray) \bcoh & *(llgray) 2^\bp \\
\end{ytableau} & \qquad
\begin{array}{|l|} \hline
\text{white = } \bp \\ \text{light gray = } 2^\bp \\ \text{dark gray = } 2^{(2^\bp)} \\
\hline
\end{array}
\end{align*}
\end{example}

\subsection{Other posets}

\begin{example} \label{eg:opposite-natural-numbers}
Let $P$ be the opposite of the poset of natural numbers. There is exactly one non-principal filter: the set $P$ itself. So $X = \{0,1,2,3,\dots \} \cup \{\infty\}$. Clearly $\bp = \bx = \bcl$. Using Proposition \ref{prop:cardinalities} we then get $\bgt = \bep = 2^\bp$. All open sets in $X$ are compact. So the patch topology agrees with the strong topology, and in particular $\bcoh = \bep$.
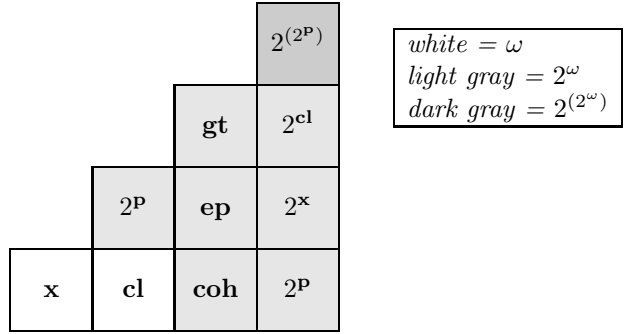
\begin{align*}
\ytableausetup{mathmode, boxsize=3em}
\begin{ytableau}
\none & \none & \none & *(lgray) 2^{(2^\bp)} \\
\none & \none & *(llgray) \bgt  & *(llgray) 2^\bcl \\
\none & *(llgray) 2^\bp  & *(llgray) \bep  & *(llgray) 2^\bx \\
\bx   &  \bcl  & *(llgray) \bcoh & *(llgray) 2^\bp \\
\end{ytableau} & \qquad
\begin{array}{|l|} \hline
\text{white = } \omega \\ \text{light gray = } 2^\omega \\ \text{dark gray = } 2^{(2^\omega)} \\
\hline
\end{array}
\end{align*}
\end{example}

\begin{example}[Power set] \label{eg:power-set}
Let $P$ be the poset of finite subsets of some infinite set $I$, with the opposite of the inclusion relation. Then $X = \P(I)$ with the inclusion relation. So $\bx = 2^\bp$ and by Proposition \ref{prop:cardinalities} this shows $\bcl = \bcoh = 2^\bp$. We still have to determine $\bep$ and $\bgt$. Note that every upwards closed set in $X$ is closed for the strong topology (equiv.~sober). Recall that an antichain in $X$ is a subset such that all elements in it are pairwise incomparable. Sending an antichain to the upwards closed subset generated by it, is an injective operation, because the antichain can be recovered as the set of minimal elements.\footnote{It is not surjective, consider for example the poset of real numbers and the upwards closed set $(0,+\infty)$.}

In order to find how many antichains there are, we use a trick that seems to be well-known (at least in the case $I = \NN$). Write $I = I_1 \sqcup I_2$ with $|I_1|=|I_2|=|I|$ and take a bijection $\psi:I_1\to I_2$. Consider the set
\begin{equation}
A = \{ x \in X ~:~ i \in x \Leftrightarrow \psi(i) \notin x  \}.
\end{equation}
Note that $A$ is an antichain: if $x \subseteq y$, $x,y \in A$, then $a \in x \Rightarrow a \in y$ but also $a \notin x \Rightarrow a \notin y$. So $x = y$. Clearly, $|A| = 2^{|I_1|} = 2^\bp$, and each subset of $A$ is again an antichain. This shows that there are at least $2^{(2^\bp)}$ antichain in $X$, and at least as many upwards closed sets. So we find $\bep = \bgt = 2^{(2^\bp)}$.
\begin{align*}
\ytableausetup{mathmode, boxsize=3em}
\begin{ytableau}
\none & \none & \none & *(lgray) 2^{(2^\bp)} \\
\none & \none & *(lgray) \bgt  & *(lgray) 2^\bcl \\
\none & *(llgray) 2^\bp  & *(lgray) \bep  & *(lgray) 2^\bx \\
*(llgray) \bx   & *(llgray) \bcl  & *(llgray) \bcoh & *(llgray) 2^\bp \\
\end{ytableau} & \qquad
\begin{array}{|l|} \hline
\text{light gray = } 2^\bp \\ \text{dark gray = } 2^{(2^\bp)} \\
\hline
\end{array}
\end{align*}
\end{example}

\begin{example}[The big cell] \label{eg:big-cell}
Let $P$ be the poset of positive natural numbers and the opposite of the division relation, so $m \leq n$ in $P$ if and only if $n | m$. Then $X = \SS$ is the poset of supernatural numbers under the division relation, see \cite[Theorem 1]{llb-covers}. To determine $\bgt$, $\bep$ and $\bcoh$, it is enough to know that there is a subposet $\P(I) \subset X$ where $I$ is the set of prime numbers. So $\bx = 2^\bp$ and this shows $\bcl = \bcoh = 2^\bp$. Each antichain in $\P(I)$ produces an antichain in $X$, so $\bep = \bgt = 2^{(2^{\omega})}$.
\begin{align*}
\ytableausetup{mathmode, boxsize=3em}
\begin{ytableau}
\none & \none & \none & *(lgray) 2^{(2^\bp)} \\
\none & \none & *(lgray) \bgt  & *(lgray) 2^\bcl \\
\none & *(llgray) 2^\bp  & *(lgray) \bep  & *(lgray) 2^\bx \\
*(llgray) \bx   & *(llgray) \bcl  & *(llgray) \bcoh & *(llgray) 2^\bp \\
\end{ytableau} & \qquad
\begin{array}{|l|} \hline
\text{light gray = } 2^\omega \\ \text{dark gray = } 2^{(2^\omega)} \\
\hline
\end{array}
\end{align*}
\end{example}

\begin{example}
Let $P$ be the poset of nonnegative real numbers with the opposite partial order. Then the filters on $P$ are $[0,r]$ or $[0,r)$ for some $r \in \RR$ or $[0,+\infty)$. In $X$ these filters will be denoted by $r$ resp.~$r_-$ resp.~$\infty$. Clearly $\bx = \bp = 2^\omega$. With a similar consideration, we see $\bcl = \bp$. A subbasis for the patch topology  on $X$ is given by the intervals $[a,+\infty)$ and their complements $[0,a_-]$ for $a \geq 0$. So a basis of the patch topology is given by the intervals $[a,b_-]$ for $a < b$, $a,b \geq 0$. Each open set $U$ for the patch topology is then a disjoint union of open, closed and half-open intervals in $X$. Picking a rational number in each path-component shows that there are only countable many path-components in $U$. So to each open set we can associate a countable subset of the set of intervals in $X$. The set of intervals in $X$ has cardinality $2^\omega$, so the set of countable subsets of it has cardinality $2^\omega$ as well. We find that $\bcoh = 2^\omega = \bx$.
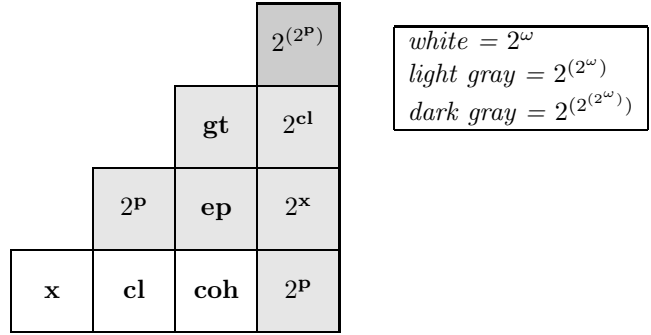
\begin{align*}
\ytableausetup{mathmode, boxsize=3em}
\begin{ytableau}
\none & \none & \none & *(lgray) 2^{(2^\bp)} \\
\none & \none & *(llgray) \bgt  & *(llgray) 2^\bcl \\
\none & *(llgray) 2^\bp  & *(llgray) \bep  & *(llgray) 2^\bx \\
\bx   &\bcl  & \bcoh & *(llgray) 2^\bp \\
\end{ytableau} & \qquad
\begin{array}{|l|} \hline
\text{white = } 2^\omega \\ \text{light gray = } 2^{(2^\omega)} \\ \text{dark gray = } 2^{(2^{(2^\omega)})} \\
\hline
\end{array}
\end{align*}
\end{example}

\begin{remark}
For each of the inequalities resulting from Proposition \ref{prop:cardinalities}, we have now given an example where the inequality is strict, with the exception of $\bep \leq \bgt$. It is unknown to the author if it is possible to have $\bep < \bgt$. 

Note that there are in each example at most three cardinalities (colors). This is no surprise: it is consistent with ZFC that there are at most three cardinalities $\bp \leq \alpha \leq 2^{(2^\bp)}$ (Generalized Continuum Hypothesis). So in ZFC it is not possible to construct an example with four colors or more.
\end{remark}

\section*{Acknowledgements}

I would like to thank Lieven Le Bruyn for discussions about Grothendieck topologies on the big cell (the special case that leaded to this paper).

\bibliographystyle{amsplainarxiv}
\bibliography{paper5}

\providecommand{\bysame}{\leavevmode\hbox to3em{\hrulefill}\thinspace}
\providecommand{\MR}{\relax\ifhmode\unskip\space\fi MR }
\providecommand{\MRhref}[2]{%
  \href{http://www.ams.org/mathscinet-getitem?mr=#1}{#2}
}
\providecommand{\href}[2]{#2}
\begin{thebibliography}{10}

\bibitem{sga4-2}
\emph{Th\'eorie des topos et cohomologie \'etale des sch\'emas. {T}ome 2},
  Lecture Notes in Mathematics, Vol. 270, Springer-Verlag, Berlin-New York,
  1972, S\'eminaire de G\'eom\'etrie Alg\'ebrique du Bois-Marie 1963--1964 (SGA
  4), Dirig\'e par M. Artin, A. Grothendieck et J. L. Verdier. Avec la
  collaboration de N. Bourbaki, P. Deligne et B. Saint-Donat.

\bibitem{amadio-curien}
Roberto~M. Amadio and Pierre-Louis Curien, \emph{Domains and lambda-calculi},
  Cambridge Tracts in Theoretical Computer Science, vol.~46, Cambridge
  University Press, Cambridge, 1998.

\bibitem{caramello-stone}
Olivia Caramello, \emph{A topos-theoretic approach to stone-type dualities},
  (2011), preprint, \href {http://arxiv.org/abs/1103.3493}
  {\path{arXiv:1103.3493}}.

\bibitem{gabber-kelly}
Ofer Gabber and Shane Kelly, \emph{Points in algebraic geometry}, J. Pure Appl.
  Algebra \textbf{219} (2015), no.~10, 4667--4680.

\bibitem{compendium-2}
G.~Gierz, K.~H. Hofmann, K.~Keimel, J.~D. Lawson, M.~Mislove, and D.~S. Scott,
  \emph{Continuous lattices and domains}, Encyclopedia of Mathematics and its
  Applications, vol.~93, Cambridge University Press, Cambridge, 2003.

\bibitem{goubault-larrecq}
Jean Goubault-Larrecq, \emph{Non-{H}ausdorff topology and domain theory}, New
  Mathematical Monographs, vol.~22, Cambridge University Press, Cambridge,
  2013, [On the cover: Selected topics in point-set topology].

\bibitem{ega3}
A.~Grothendieck, \emph{{\'El\'ements de g\'eom\'etrie alg\'ebrique. {IV}.
  \'Etude locale des sch\'emas et des morphismes de sch\'emas. {III}}}, Inst.
  Hautes \'Etudes Sci. Publ. Math. (1966), no.~28, 255, URL:
  \url{http://www.numdam.org/item?id=PMIHES_1966__28__255_0}.

\bibitem{azutop}
Jens Hemelaer, \emph{Azumaya toposes},  (2017), preprint, \href
  {http://arxiv.org/abs/1707.03814} {\path{arXiv:1707.03814}}.

\bibitem{azureps}
Jens Hemelaer and Lieven Le~Bruyn, \emph{Azumaya representation schemes},
  (2016), preprint, \href {http://arxiv.org/abs/1606.07885}
  {\path{arXiv:1606.07885}}.

\bibitem{hochster}
Melvin Hochster, \emph{Prime ideal structure in commutative rings}, Trans.
  Amer. Math. Soc. \textbf{142} (1969), 43--60.

\bibitem{johnstone-topos}
Peter~T. Johnstone, \emph{Topos theory}, Academic Press [Harcourt Brace
  Jovanovich, Publishers], London-New York, 1977, London Mathematical Society
  Monographs, Vol. 10.

\bibitem{keimel-lawson}
Klaus Keimel and Jimmie~D. Lawson, \emph{D-completions and the {$d$}-topology},
  Ann. Pure Appl. Logic \textbf{159} (2009), no.~3, 292--306.

\bibitem{llb-covers}
Lieven Le~Bruyn, \emph{Covers of the arithmetic site},  (2016), preprint, \href
  {http://arxiv.org/abs/1602.01627v1} {\path{arXiv:1602.01627v1}}.

\bibitem{lindenhovius}
Bert Lindenhovius, \emph{Grothendieck topologies on a poset},  (2014),
  preprint, \href {http://arxiv.org/abs/1405.4408v2}
  {\path{arXiv:1405.4408v2}}.

\bibitem{lindenhovius-thesis}
\bysame, \emph{C(a)}, 2016, Thesis (Ph.D.), Radboud University Nijmegen, URL:
  \url{www.math.ru.nl/~landsman/Lindenhovius.pdf}.

\bibitem{maclane-moerdijk-sheaves}
Saunders Mac~Lane and Ieke Moerdijk, \emph{Sheaves in geometry and logic},
  Universitext, Springer-Verlag, New York, 1994, A first introduction to topos
  theory, Corrected reprint of the 1992 edition.

\bibitem{nerode}
Anil Nerode, \emph{Some {S}tone spaces and recursion theory}, Duke Math. J.
  \textbf{26} (1959), 397--406.

\bibitem{priestley}
Hilary~A. Priestley, \emph{Spectral sets}, J. Pure Appl. Algebra \textbf{94}
  (1994), no.~1, 101--114.

\bibitem{schroeer}
Stefan Schr\"oer, \emph{Points in the fppf topology}, Ann. Sc. Norm. Super.
  Pisa Cl. Sci. (5) \textbf{17} (2017), no.~2, 419--447.

\bibitem{stacks-project}
The {Stacks project authors}, \emph{The stacks project},
  \url{https://stacks.math.columbia.edu}, 2018.

\bibitem{vickers}
Steven Vickers, \emph{Topology via logic}, Cambridge Tracts in Theoretical
  Computer Science, vol.~5, Cambridge University Press, Cambridge, 1989.

\end{thebibliography}
\end{document}